\documentclass[10pt,reqno,oneside]{amsart}
\usepackage[left=3cm,right=3cm,top=4cm,bottom=3cm]{geometry}
\usepackage{amsmath,amsfonts,amsthm,amssymb,bbm,color}
\usepackage{amsmath,amssymb,stmaryrd,bbm,latexsym,color}
\usepackage{pdfsync}
\usepackage[citecolor=blue,colorlinks=true]{hyperref}
\usepackage{mathrsfs}
\usepackage{upref}
\usepackage{enumitem}

\numberwithin{equation}{section}

\newcommand{\dif}{\mathrm{d}}
\newcommand{\dd}{\mathrm{d}}
\newcommand{\mathd}{\mathrm{d}}

\newcommand{\ep}{\varepsilon}

\newcommand{\al}{\alpha}

\newcommand{\vp}{\varphi}

\renewcommand{\leq}{\leqslant}
\renewcommand{\geq}{\geqslant}

\newcommand{\R}{\mathbb R}
\newcommand{\N}{\mathbb N}

\newcommand{\cl}{\mathcal L}

\DeclareMathOperator{\diver}{div}

\newtheorem{theorem}{Theorem}[section]

\newtheorem{lemma}[theorem]{Lemma}

\theoremstyle{definition}
\newtheorem{definition}[theorem]{Definition}

\theoremstyle{remark}

\theoremstyle{remark}

\newcommand{\assign}{:=}

\newcommand{\tmop}[1]{\ensuremath{\operatorname{#1}}}

\begin{document}

\title[On the rough Gronwall lemma and its applications]{On the rough Gronwall lemma and its applications}

\author{Martina Hofmanov\'a}
\address[M. Hofmanov\'a]{Institute of Mathematics,  Technical University Berlin, Stra\ss e des 17. Juni 136, 10623 Berlin, Germany}
\email{hofmanov@math.tu-berlin.de}

\thanks{Financial support by the DFG via Research Unit FOR 2402 is gratefully acknowledged.}

\begin{abstract}
We present a rough path analog of  the classical  Gronwall Lemma   introduced recently by A.~Deya, M.~Gubinelli, M.~Hofmanov\'a, S.~Tindel in~[\href{https://arxiv.org/abs/1604.00437}{arXiv:1604.00437}] and discuss two of its applications. First, it is applied in the framework of rough path driven PDEs in order to establish energy estimates for weak solutions. Second, it is used in order to prove uniqueness for reflected rough differential equations.
\end{abstract}



\date{\today}

\maketitle

\section{Introduction}

The theory of  rough paths was introduced by Terry Lyons in his seminal work \cite{MR1654527}. It can be briefly described as an extension of the classical theory of controlled differential equations which is robust enough to allow for a deterministic treatment of stochastic differential equations. To be more precise and in order to fix the ideas, let us consider a controlled differential equation of the form
\begin{equation}\label{eq:}
\dif y_t=f(y_t)\dif x_t,
\end{equation}
where the driving signal $x$ possesses only a limited regularity, namely, it does not have finite variation. One requirement for any sensible (deterministic) theory of such equations is certainly robustness: approximating $x$ by smooth paths, say $(x^n)$, and solving the corresponding approximate problems
$$\dif y^n_t=f(y^n_t)\dif x^n_t,$$
we expect that  solutions $y^n$ remain close, in a suitable sense, to a (unique) solution $y$ of the original problem \eqref{eq:}. In other words, we look for suitable topologies that render the solution map $x\mapsto y$ continuous, while allowing for driving signals of low regularity. The first positive answer to this problem was given in 1936 by Young  \cite{young}, who introduced an extension of the Stieltjes integral for paths of finite $p$-variation with $p<2$. Roughly speaking, he showed that  one can define an integral of $y$ against $x$ such that the mapping
$$V^p([0,T])\times V^p([0,T])\to V^p([0,T]),\qquad (x,y)\mapsto \int_0^\cdot y_s\,\dif x_s$$
is continuous  provided $p<2$, where $V^p([0,T])$ stands for the space of continuous paths with finite $p$-variation on $[0,T]$.

However, one of the most interesting examples of a driving signal, namely, the Brownian motion, is not covered by this result. Indeed, it can be shown that its sample paths have unbounded $p$-variation for any $p\leq 2$.
Moreover, it turns out that the situation is even more delicate: Lyons \cite{Lyo91} proved that there exists \emph{no }Banach space $\mathcal{B}$ containing sample paths of Brownian motions 
such that the map
$$(x,y)\mapsto \int_0^\cdot y_t\,\dot{x}_t\,\dif t$$
defined on smooth functions extends to a continuous map on $\mathcal{B}\times\mathcal{B}$.
The breakthrough by Lyons \cite{MR1654527} was then based on the insight that an important part of information on the signal $x$  is missing, due to its low regularity. In particular, he showed that if the driving signal $x$ has finite $p$-variation for $p\in [2,3)$, the continuity of the solution map to \eqref{eq:} can be recovered by enhancing $x$ by a second component $X^2$, the so-called L\'evy's area, which corresponds to the iterated integral
$$X^2_{st}=:\int_s^t (x_r-x_s)\otimes\dif x_r.$$
Consequently, the solution map to \eqref{eq:}, that is $(x,X^2)\mapsto y$, can be shown to be continuous with respect to appropriate topologies.

Remark that the above lines shall be understood as follows: apart from the path $x$ we are given another  datum $X^2$ which satisfies certain analytic and algebraic properties and which plays the role of the iterated integral $x$ against $x$. Note that for instance in the case of a Brownian motion we still face the same issue as before and the iterated integral cannot be defined by deterministic arguments. However, the striking advantage now is that the only integral that  needs to be constructed is the iterated integral of the Brownian motion itself and this can be done via probabilistic arguments. Therefore, we are able to separate the analytic and probabilistic part of solving stochastic differential equations. In the first step, we use probability (or any other available tool for the particular driving signal at hand) to construct the iterated integral. In the second step, we fix one realization of the process and its iterated integral and proceed deterministically.

\smallskip

Since its introduction, the rough path theory has found a large number of applications and tremendous progress has been made in application of rough path ideas to ordinary as well as partial differential equations driven by rough signals. We refer the reader for instance to the works by Friz et al. \cite{CF,CFO}, Gubinelli--Tindel \cite{P1:DGT,P1:GLT,P1:GT}, Gubinelli--Imkeller--Perkowski \cite{GIP14}, Hairer \cite{ha14} for a tiny sample of the exponentially growing literature on the subject. In view of these exciting developments it is remarkable that many basic PDE methods have not yet found their rough path analogues. For instance, until recently it was an open problem how to construct (weak) solutions to RPDEs using energy methods.

In \cite{BG,DGHT} we started a long term research programme where these questions will be addressed.
One of the aims is to develop a theory applicable to a wide class of RPDEs by following the standard PDE strategies in order to obtain  existence and uniqueness results. 
In \cite{DGHT}, we introduced the general framework and developed innovative a priori estimates based on a new rough Gronwall lemma argument. The theory was applied to conservation laws with rough flux. Moreover, these new techniques already proved sufficiently flexible  and useful  as in \cite{DGHT2} we were able to establish uniqueness for reflected rough differential equations, a problem which remained open in the literature as a suitable Gronwall lemma in the context of rough paths was missing.
The aim of the present paper is to discuss the main ideas of \cite{DGHT,DGHT2} in simple terms  and to make the link between the two applications. It is expected that the framework will find further applications in future.

\section{Intrinsic notion of solution}

\subsection{Notation}

First of all, let us recall the definition of the increment operator, denoted by $\delta$. If $g$ is a path defined on $[0,T]$ and $s,t\in[0,T]$ then $\delta g_{st}:= g_t - g_s$, if $g$ is a $2$-index map defined on $[0,T]^2$ then $\delta g_{sut}:=g_{st}-g_{su}-g_{ut}$. For two quantities $a$ and $b$ the relation $a\lesssim_{x} b$ means $a\leq c_{x} b$, for a constant $c_{x}$ depending on a (possibly multidimensional) parameter $x$.

In the sequel, given an interval $I$ we call a \emph{control on $I$} (and denote it by $\omega$) any superadditive map on $\Delta_I:=\{(s,t)\in I^2: \ s\leq t\}$, that is, any map $\omega: \Delta_I \to [0,\infty[$ such that, for all $s\leq u\leq t$,
$$\omega(s,u)+\omega(u,t) \leq \omega(s,t).$$
We will say that a control is \emph{regular} if $\lim_{|t-s|\to 0} \omega(s,t) = 0$. Also, given a control $\omega$ on an interval $I=[a,b]$, we will use the notation $\omega(I):=\omega(a,b)$. 
Given a time interval $I$, a parameter $p>0$, a Banach space $E$  we denote by $\overline V^p_1(I;E)$ the space of finite $p$-variation functions taking values in $E$. The corresponding seminorm is denoted by
$$\|g\|_{\overline V^p_1(I;E)}:=\sup_{(t_i) \in \mathcal{P}(I)} \left( \sum_{i} |g_{t_i} - g_{t_{i+1}}|^p \right)^{\frac{1}{p}},$$
where $\mathcal{P}(I)$ denotes the set of all partitions of the interval $I$.
If the right hand side is finite then
$$\omega_g(s,t) = \|g\|_{\overline V^p_1(I;E)}^p$$ defines a control on $I$, and we denote by $ V^p_1(I;E)$ the set of elements $g\in \overline V^p_1(I;E)$ for which $\omega_g$ is regular on $I$. 
We denote by $\overline V^p_2(I;E)$ the set of two-index maps $g : I\times I \to E$ with left and right limits in each of the variables and for which there exists a control $\omega$ such that
$$
|g_{st}| \leq \omega(s,t)^{\frac{1}{p}}
$$
for all $s,t\in I$. We also define the space $\overline V^p_{2,\text{loc}}(I;E)$ of maps $g:I\times I\to E$ such that there exists a countable covering $\{I_k\}_k$ of $I$ satisfying $g \in \overline V^p_2(I_k;E)$ for any $k$. We write $g\in  V^p_2(I;E)$ or $g\in  V^p_{2,\text{loc}}(I;E)$ if the control can be chosen regular.

\subsection{Rough drivers}

To begin with, let us introduce the notion of a rough path. For a thorough introduction to the theory of rough paths we refer the reader to the monographs \cite{{friz_course_2014}, FV, MR2314753}.

\begin{definition}\label{defi-rough-path}
Let $d \in\N $, $p\in[2,3)$. A continuous $p$-rough path is a  pair 
\begin{equation}\label{p-var-rp}
X=(X^1,X^2) \in V^p_2 ([0,T];\mathbb{R}^{d}) \times V^{p/2}_2 ([0,T]; \mathbb{R}^{d\times d}) 
\end{equation}
that satisfies  Chen's relation 
\begin{equation*}\label{chen-rela}
  \delta X^{2}_{sut}=X^{1}_{su} \otimes X^{1}_{ut} \ , \qquad s\leq u\leq t\in [0,T]  .
\end{equation*}
A rough path $X$ is said to be geometric if it can be obtained as the limit in the $p$-variation topology given in (\ref{p-var-rp}) of a sequence of rough paths  $X^\ep=(X^{\ep,1},X^{\ep,2})$ explicitly defined as
$$X^{\ep,1}_{st}\assign \delta x^{\ep}_{st} \ , \qquad X^{\ep,2}_{st}\assign \int_s^t \delta x^{\ep}_{su} \otimes \mathd x^{\ep}_u \ ,$$
for some smooth path $x^\ep:[0,T] \to \mathbb{R}^d$. 
\end{definition}

Throughout this paper, we are only  concerned with geometric rough paths. A pleasant advantage is a first order chain rule formulae similar to the one known for smooth paths. Recall that this is not true within the It\^o stochastic integration theory, where only a (second order) It\^o formula is available. However, for the Stratonovich stochastic integral, the first order chain rule holds true. Thus  in case of a Brownian motion we  employ Stratonovich integration for the  construction of the iterated integrals of a geometric rough path, whereas the It\^o integral leads to nongeometric setting.

We are now in a position to provide a clear interpretation of the controlled equation \eqref{eq:} in the rough path setting.

\begin{definition}\label{def:sol-rrde1}
Let  $y_{\text{in}}\in \R^N$, let $f :\mathbbm{R}^N \to \mathcal{L} (\mathbbm{R}^d ;\mathbbm{R}^N)$ be a differentiable function and let $X=(X^1,X^2)$ be a continuous $p$-rough path with $p\in[2,3)$. A path $y\in  V^p_1 ([0, T] ; \mathbbm{R}^d) $  solves the 
rough differential equation
$$
\dif y = f(y)\,\dif X,\qquad y_0=y_{\text{in}},
$$
if there exists a 2-index map $y^\natural \in V^{p/3}_{2, \text{loc}}([0,T];\mathbbm{R}^N)$ such that for all $s,t\in [0,T]$, we have
\begin{equation}\label{eq:intr}
   \delta y_{s t} = f (y_s) {X}^{1}_{s t} + f_{2} (y_s)
{X}^{2}_{s t} +y^\natural_{st}, \qquad y_0=y_{\text{in}},
\end{equation}
where we have set $f_{2,ij} \assign f'_i f_j $.
\end{definition}

This formulation has an intuitive appeal, however note that more precisely it should  be understood   as saying that $y$ is a solution of the equation if 
\begin{equation*}
y^{\natural}_{s t}: =  \delta y_{s t} - f (y_s) {X}^{1}_{s t} -f_{2} (y_s)
{X}^{2}_{s t}
 \end{equation*}
belongs to $V^{p/3}_{2,\text{loc}}([0,T],\R^N)$. In other words, for every path $y$ one can define $y^\natural$ by the above formula, however, this would in general not be a remainder, that is, it would not have the required time regularity.

In the previous example of a rough differential equation we considered a noise term which was irregular in time and its coefficient contained nonlinear but bounded dependence on the solution. Nevertheless, in the PDE theory one is often lead to differential, i.e. unbounded, operators. As a model example, we may think of a rough transport equation
\begin{equation}\label{eq:tr}
\dif u=V\cdot\nabla u\,\dif X,\qquad u_0=u_{\text{in}}.
\end{equation}
In order to treat these equations, we introduce the notion of an unbounded rough driver. It can be regarded as an operator valued rough path taking values in a suitable space of unbounded operators.

In what follows, we call a \emph{scale} any sequence $\big(E_n,\lVert\cdot\rVert_n\big)_{n\in\N_0}$ of Banach spaces such that $E_{n+1}$ is continuously embedded into $E_n$. Besides, for $n\in\N_0$ we denote by $E_{-n}$ the topological dual of $E_n$. Note that it is necessary to distinguish between $E_0$ and its dual $E_{-0}$ as they are generally different spaces.

\begin{definition}
\label{def:urd}
Let $p\in [2,3)$ be given. A continuous unbounded $p$-rough driver with respect to the scale $\big(E_n,\lVert\cdot\rVert_n\big)_{n\in\N_0}$, is a pair ${A} = \big(A^1,A^2\big)$ of $2$-index maps such that
$$A^1_{st}\in  \cl(E_{-n},E_{-(n+1)}) \ \ \text{for} \ \ n\in \{0,2\}, 
\qquad  A^2_{st} \in \cl(E_{-n},E_{-(n+2)}) \ \  \text{for} \ \ n\in \{0,1\},$$
and there exists a continuous control $\omega_A$ on $[0,T]$ such that for every $s,t\in [0,T]$,
\begin{align*}
\lVert A^1_{st}\rVert_{\cl(E_{-n},E_{-(n+1)})}^p &\leq\omega_A(s,t) \qquad \text{for}\ \ n\in \{0,2\}, \\
\lVert A^2_{st}\rVert_{\cl(E_{-n},E_{-(n+2)})}^{p/2} &\leq\omega_A(s,t) \qquad \text{for}\ \ n\in \{0,1\},
\end{align*}
and, in addition, Chen's relation holds true, that is,
\begin{equation*}
\delta A^1_{sut}=0,\qquad\delta A^2_{sut}=A^1_{ut}A^1_{su},\qquad\text{for all}\quad 0\leq s\leq u\leq t\leq T.
\end{equation*}
\end{definition}

It can be checked that under sufficient regularity assumptions on the family of vector fields $V$, the transport noise in \eqref{eq:tr} defines an unbounded $p$-rough driver  in the scale $W^{n,2}(\R^N)$ by
\begin{equation*}
A^1_{st}u:=X^{1,k}_{st} \,V^k \cdot\nabla u,\qquad A^2_{s t}u:= X^{2,jk}_{st}\, V^k\cdot\nabla (V^j\cdot\nabla u).
\end{equation*}
Here we employ the Einstein summation convention over repeated indices.
Consequently, in analogy to Definition \ref{def:sol-rrde1} we may formulate the notion of weak solution to \eqref{eq:tr} as follows.

\begin{definition}\label{def:gen}
Let $A = \big(A^1,A^2\big)$ be a continuous unbounded $p$-rough driver with respect to the scale $(E_n)$. A path $u:[0,T] \to E_{-0}$ is  a weak solution to \eqref{eq:tr} provided there exist $u^\natural \in V^{p/3}_{2,\text{loc}}([0,T],E_{-3})$ such that, for every $s,t\in [0,T]$, $s<t$, and every test function $\varphi\in E_3$ it holds
\begin{equation}\label{eq:gen2a}
(\delta u)_{st}(\varphi)=u_s(\{A^{1,*}_{st}+A^{2,*}_{st}\}\varphi)+u^\natural_{st}(\varphi), \qquad u_0=u_{\text{in}}.
\end{equation}

\end{definition}

The main difficulty in working with the intrinsic formulations \eqref{eq:intr} and \eqref{eq:gen2a} lies in the remainders $y^\natural$ and $u^\natural$, respectively. Indeed, all the other terms  are explicit and do not even contain any time integration. Therefore, estimations of these terms are very straightforward. On the other hand, the only information available so far on the remainders $y^\natural$ and $u^\natural$   are the respective equations \eqref{eq:intr} and \eqref{eq:gen2a} and they have to be carefully investigated in order to obtain a priori estimates for the solutions $y$ and $u$.

\section{A priori estimates for rough partial differential equations}

In view of the possible applications, it is necessary to allow for (deterministic) drift terms. One could then formulate equations in the general form
\begin{equation}\label{eq:gen2}
(\delta g)_{st}(\varphi)=(\delta \mu)_{st}(\varphi)+g_s(\{A^{1,*}_{st}+A^{2,*}_{st}\}\varphi)+g^\natural_{st}(\varphi),
\end{equation}
with the drift being for instance another transport term or a second order (possibly nonlinear) elliptic operator, i.e.
$$
\mu(\dif t)=V_0\cdot\nabla g\,\dif t,\qquad \mu(\dif t)=\diver(A(x,g,\nabla g)\nabla g)\,\dif t.
$$
In this case, the rule of thumb is to integrate whenever possible, that is, we do not consider a local approximation by Riemann sums of the drift  in \eqref{eq:gen2} (as we do for the rough integral) but rather the increments of the integral itself.

The key result establishing the a priori bounds on the remainder $g^\natural$ reads as follows. The proof of its   more general form is presented in \cite[Theorem 2.5]{DGHT}.

\begin{theorem}\label{prop:apriori}
Let $I$ be a subinterval of $[0,T]$. Consider a path $\mu \in \overline V_1^1(I;E_{-1})$ satisfying for some control $\omega_\mu$ and for every $\varphi \in E_1$
\begin{equation*}\label{cond-mu-n}
 |(\delta \mu)_{st}(\varphi)|\leq \omega_\mu(s,t)\, \|\varphi\|_{E_1} .
\end{equation*}
Let  $g$ be a solution of the equation \eqref{eq:gen2} on $I$ such that $g$ is controlled over the whole interval $I$, that is, $g^\natural\in V_2^{p/3}(I,E_{-3})$.
Then there exists a constant $L>0$ such that if $\omega_A(I)\leq L$ then for all $s<t\in I$
\begin{equation}\label{apriori-bound} 
\begin{split}
 \|g^{\natural}_{st}\|_{E_{-3}}
 \lesssim_{A,I}  \|g\|_{L^\infty(s,t;E_{-0})}\,\omega_A(s,t)^{\frac{3}{p}}+\omega_\mu(s,t)\omega_A(s,t)^{\frac{3-p}{p}}.
\end{split}
\end{equation}
\end{theorem}

Let us explain the application of the above result on the example of a heat equation with transport noise
$$
\dif u=\Delta u\,\dif t+V\cdot \nabla u\,\dif X.
$$
The ultimate goal is to derive the energy estimates for the solution known from the classical (deterministic) PDE theory. That is, we intend to show that $u$ belongs to $ L^\infty(0,T;L^2(\R^N))\cap L^2(0,T;W^{1,2}(\R^N))$. To this end, one first derives the equation for $u^2$, which can be done rigorously for instance on the level of smooth approximations. It can  be formulated in the scale $E_n=W^{n,\infty}(\R^N)$  as
\begin{align}\label{u2}
\begin{aligned}
\delta  u^2(\varphi )_{s t} &=-2\int_{s}^{t} 
      |\nabla u_r|^2 ( \varphi)  \,\dd r-2\int_s^t(u\nabla u_r)(\nabla \vp)\,\dd r+ u^2_{s} (A^{1,\ast}_{s t} \varphi ) +  u^2_{s} (A^{2, \ast}_{s t} \varphi
     ) + u^{2,\natural}_{s t}( \varphi ).
\end{aligned}
\end{align}
Theorem \ref{prop:apriori} then applies with
$$
\omega_\mu(s,t)=\int_s^t \|\nabla u_r\|_{L^2}^2 \,\dd r+\bigg(\int_s^t\|\nabla u_r\|_{L^2}^2\,\dd r \bigg)^\frac{1}{2}\bigg(\int_s^t\| u_r\|_{L^2}^2\,\dd r \bigg)^\frac{1}{2}.
$$
This is the core of our rough Gronwall lemma argument, which is then concluded using the following result, whose proof can be found in \cite[Lemma 2.7]{DGHT}.

\begin{lemma}[Rough Gronwall Lemma]\label{lemma:basic-rough-gronwall}
Fix a time horizon $T>0$ and let $G:[0,T] \to [0,\infty)$ be a path such that for some constants $C,L>0$, $\kappa\geq 1$ and some controls $\omega_1,\omega_2$ on $[0,T]$ with $\omega_1$ being  regular, one has
\begin{equation*}
\delta G_{st} \leq C \Big(\sup_{0\leq r\leq t} G_r \Big) \, \omega_1(s,t)^{\frac{1}{\kappa}} + \omega_2(s,t),
\end{equation*}
for every $s<t\in [0,T]$ satisfying $\omega_1(s,t) \leq L$. Then it holds
$$\sup_{0\leq t\leq T} G_t \leq 2 \exp\Big( \frac{\omega_1(0,T)}{\al L}\Big) \cdot \Big\{ G_0+\sup_{0\leq t\leq T}\Big(\omega_2(0,t) \, \exp\Big(-\frac{\omega_1(0,t)}{\al L}\Big)\Big) \Big\} ,$$
where $\al$ is defined as 
\begin{equation*}
\al=\min\left(1, \, \frac{1}{L (2Ce^2)^{\kappa} }\right).
\end{equation*}
\end{lemma}

Finally, we have all in hand to derive the desired a priori estimate for a solution to the above heat equation.
Indeed, taking the test function $\varphi\equiv 1$ in  \eqref{u2} and applying Theorem \ref{prop:apriori} we  obtain
\begin{align*}
&(\delta\|u\|_{L^2}^2)_{st}+2\int_s^t\|\nabla u_r\|_{L^2}^2\,\dd r
\lesssim \sup_{s\leq r\leq t}\|u_r\|_{L^2}^2\,\omega_A(s,t)^{\frac{1}{p}}
+\omega_\mu(s,t) \omega_A(s,t)^{\frac{3-p}{p}}\\
&\qquad\lesssim \sup_{s\leq r\leq t}\|u_r\|_{L^2}^2\bigg[\omega_A(s,t)^{\frac{1}{p}}+|t-s|\,\omega_A(s,t)^{\frac{3-p}{p}}\bigg]+\int_s^t \|\nabla u_r\|_{L^2}^2 \,\dd r\,\omega_A(s,t)^{\frac{3-p}{p}}.
\end{align*}
Hence Lemma \ref{lemma:basic-rough-gronwall} applies with
$$G_t:=\|u_t\|_{L^2}^2+2\int_0^t \|\nabla u_r\|_{L^2}^2\,\dd r,\quad\omega_1(s,t):=\omega_A(s,t)^{\frac{\kappa}{p}}+|t-s|^{\kappa}\omega_A(s,t)^{\frac{(3-p)\kappa}{p}}+\omega_A(s,t)^{\frac{(3-p)\kappa}{p}},$$
$\omega_2(s,t)=0$, and $\kappa = \min(p,p/(3-p)) \geq 1$, and yields
\begin{align*}
\sup_{0\leq t\leq T}\|u_t\|_{L^2}^2+\int_0^T\|\nabla u_t\|_{L^2}^2\,\dd t\lesssim \exp\Big( \frac{\omega_1(0,T)}{\al L}\Big) \, \|u_0\|_{L^2}^2.
\end{align*}
With these a priori estimates one can immediately proceed to the proof of existence of a weak solution. It does not rely on the Banach fixed point argument but rather on compactness of suitable approximate solutions. Therefore we are able to separate the proof of existence from the proof of uniqueness, which is needed for many problems of interest and in particular for reflected rough differential equations discussed in the following section.

Let us point out that at the current stage we are not able to treat PDEs with nonlinear noise terms, such as
$$\dif u =\Delta u\,\dif t+ f(u)\,\dif X,$$
where $f$ is a sufficiently regular function. Even though such a nonlinear noise can be treated in the finite dimensional framework, for instance as shown in Section \ref{sec:uniq}, in the above PDE setting it is an open question how to derive energy estimates from the corresponding variational formulation.

\section{Uniqueness for reflected rough differential equations}
\label{sec:uniq}

Let us present another application of the above rough Gronwall lemma argument from a seemingly rather different field. We are interested in the one-dimensional RDE reflected at $0$, which can be described as follows: given a time $T > 0$, a smooth function $f :\mathbbm{R} \to \mathcal{L} (\mathbbm{R}^d ;\mathbbm{R})$ and a continuous $p$-rough path ${X}$ with $p\in[2,3)$, find an  path $y\in V^p_1 ([0, T];\mathbbm{R}_{\geqslant 0})$ and an  increasing function (the so-called reflection measure) $m\in  V^1_1 ([0, T];\mathbbm{R}_{\geqslant 0})$ that together satisfy
\begin{equation}\label{rrde-informal}
\mathd y _t = f (y_t)\, \mathd {X}_t + \mathd m_t  , \qquad y_t \,\mathd m_t=0 .
\end{equation}
Thus, the idea  is to exhibit a path $y$ that somehow follows the dynamics in \eqref{eq:}, but is also forced to stay nonnegative thanks to the intervention of some regular {local time} $m$ at $0$.

This problem was studied  by Aida \cite{aida_reflected_2015,aida-2016}, however, only existence of a solution  was established and the uniqueness issue was left open. The main difficulty lies in the lack of regularity of the corresponding Skorokhod map which does not allow to treat the problem via the Banach fixed point theorem. Consequently, it is necessary to study existence and uniqueness separately and a suitable Gronwall lemma argument is needed to establish uniqueness.
It turns out that the above introduced framework is well suited for this task: the existence can be proved with significant simplifications and, in addition, uniqueness follows from the rough Gronwall lemma argument by contradiction.

To be more precise, we formulate the problem as follows.

\begin{definition}\label{def:sol-rrde}
Let $y_{\text{in}}> 0$, let  $f :\mathbbm{R} \to \mathcal{L} (\mathbbm{R}^d ;\mathbbm{R})$ be a differentiable function and let $X$ be a continuous $p$-rough path with $p\in[2,3)$. A pair $(y, m)\in  V^p_1 ([0, T] ; \mathbbm{R}_{\geqslant 0}) \times V^1_1 ([0, T] ;\mathbbm{R}_{\geqslant 0})$ solves the problem (\ref{rrde-informal}) on $[0,T]$ with initial condition $y_{\text{in}}$ if there exists a 2-index map $y^\natural \in V^{p/3}_{2, \text{loc}}([0,T];\mathbbm{R})$ such that for all $s,t\in [0,T]$, we have
\begin{equation}
\begin{split}
   \delta y_{s t} = f (y_s) {X}^{1}_{s t} + f_{2} (y_s)
{X}^{2}_{s t} + \delta m_{s t} +y^\natural_{st},\qquad y_0=y_{\text{in}}, \qquad  m_t=\int_0^t \mathbf{1}_{\{ y_u=0\}}  \mathd m_u ,
\end{split}
\end{equation}
where we have set $f_{2,ij} \assign f'_i f_j $ and $m([0,t]):=m_t$.
\end{definition}

With this interpretation in hand, the well-posedness result proved in \cite[Theorem 4]{DGHT2} reads as follows.

\begin{theorem} \label{th:main-wellposedness}
If $f\in \mathcal{C}_b^3(\mathbbm{R};\mathcal{L}(\mathbbm{R}^d,\mathbbm{R}))$, then problem~\eqref{rrde-informal} admits a unique solution.
\end{theorem}

\begin{proof}[Sketch of the proof of uniqueness]
The proof proceeds in several steps, we will only present the main ideas while omitting the technical details. We refer the reader to \cite[Theorem 5]{DGHT2} for the complete proof.

Let us consider two solutions, say $(y,\mu)$ and $(z,\nu)$. In order to show that they coincide, we naturally look at the equation satisfied by the difference $y-z$. Denote $Y=(y,z)$ and let $\varphi:\R\to\R_{\geq0}$ be a smooth function. Then a direct computation via Taylor expansion shows for  $h(Y):=\varphi(y-z)$ that
  \begin{equation}
    \delta h (Y)_{s t} = H_i (Y_s) {X}^{1,i}_{s t} + H_{2,ij} (Y_s)
{X}^{2,ij}_{s t}  + \int_s^t \varphi' (y_u - z_u) (\mathd \mu_u - \mathd \nu_u)+ h^{\natural}_{s t},
    \label{eq:h-dyn}
  \end{equation}
  where $h^{\natural}$ is a map in $V^{p/3}_{2}([0,T]; \mathbbm{R})$ and where we have set for all $Y=(y,z)\in \mathbbm{R}^2$
\begin{align*}\label{eq:c}
\begin{aligned}
H_i (Y) &\assign \varphi' (y - z) (f_i (y) - f_i (z)) ,\\
H_{2,ij} (Y) &\assign\varphi' (y - z) (f_{2,ij} (y) - f_{2,ij} (z)) 
+\varphi'' (y - z) (f_i (y) - f_i (z))(f_j (y) - f_j (z))  . 
\end{aligned}
\end{align*}
Our  goal is to deduce a corresponding formula for $\varphi(\xi)=|\xi|$ and to finally show with the help of the rough Gronwall lemma that $|y_t-z_t|$ vanishes for all $t\in[0,T].$ However, since the absolute value function is not smooth at zero, \eqref{eq:h-dyn} does not apply directly and it is necessary to consider a smooth approximation first and then pass to the limit. Namely, we define $\varphi_\varepsilon(\xi)=\sqrt{\varepsilon+|\xi|^2}$ and note that
 \[ | \varphi_{\varepsilon}' (\xi) | \leqslant 1\ , \qquad |
     \varphi_{\varepsilon}'' (\xi) | \leqslant 1 / \sqrt{\varepsilon^2 + | \xi
     |^2} \ , \qquad | \varphi_{\varepsilon}''' (\xi) | \leqslant 3 /
     (\varepsilon^2 + | \xi |^2). \]
Even though the second and the third derivative blow up as $\varepsilon\to 0$, a detailed computation  in the proof of  \cite[Theorem 5]{DGHT2} shows that in order to obtain estimates uniform in $\varepsilon$, it is enough to control the quantity
\begin{equation*}\label{defi:norm-trip}
 \sup_{\varepsilon\in (0,1)}\sup_{y, z \in \R} \big(| \varphi_\varepsilon' (y - z) | + | y - z
     | | \varphi_\varepsilon'' (y - z) | + | y - z |^2 | \varphi_\varepsilon''' (y - z) |\big) .
\end{equation*}
Since this is indeed finite, we obtain estimates uniform in $\varepsilon$ and finally  pass to the limit as $\varepsilon\to0$.
This ensures the existence of a limiting remainder $\Phi^{\natural}$  and we get that the path $\Phi (Y) \assign | y - z |$ satisfies
\begin{equation*}\label{eq:phi-y}
\delta \Phi (Y)_{st} = \Psi_i (Y_s) {X}^{1,i}_{st} + \Psi_{2,ij} (Y_s) {X}^{2,ij}_{st} - \omega_M (s, t) + \int_s^t \mathbf{1}_{\{y_u = z_u\}}\mathd (\mu_u + \nu_u)+\Phi^{\natural}_{st},
\end{equation*}
where 
  \[ \Psi_i (Y) \assign \tmop{sgn} (y - z) (f_i (y) - f_i (z)) , \qquad \Psi_{2,ij} (Y) \assign \tmop{sgn} (y - z) (f_{2,ij} (y) - f_{2,ij} (z))  \]
  and
  \[
  \omega_M (s, t)\assign\| \mu\|_{\overline{V}^1_1 ([s,t] )} +\|\nu\|_{\overline{V}^1_1 ([s,t] )}.\]
In addition, further estimations in the spirit of Theorem \ref{prop:apriori} show  an apriori estimate for the remainder $\Phi^\natural $ of the form
$$|\Phi^\natural_{st}|\lesssim\|y-z\|_{L^\infty(s,t)}\,\omega_1(s,t)^{3/p}+\omega_M(s,t)\omega_2(s,t)^{1/p},$$
for some regular controls $\omega_1,\omega_2$.
Consequently, we are in a position to apply the Rough Gronwall Lemma \ref{lemma:basic-rough-gronwall}  and assert that
		  \[ \sup_{r\in [s, t]} |y_r-z_r| + \omega_M (s, t) \lesssim | y_s-z_s| +
     \int_s^t \mathbf{1}_{\{y_u = z_u\}} (\mathd \mu_u + \mathd \nu_u) . \]
		
		\smallskip
		
\noindent
Assume now that $[s, t]$ is an interval where $y \neq z$ in $(s, t)$ but $y_s
  = z _s$. Then
  \[ \sup_{r\in [s, t]} |y_r-z_r| + \omega_M (s, t) \leqslant 0 \]
  which implies that $\sup_{r\in [s, t]} |y_r-z_r| = 0$ everywhere so we find a
  contradiction and such interval cannot exist. This concludes the proof of
  uniqueness.
\end{proof}

\def\ocirc#1{\ifmmode\setbox0=\hbox{$#1$}\dimen0=\ht0 \advance\dimen0
  by1pt\rlap{\hbox to\wd0{\hss\raise\dimen0
  \hbox{\hskip.2em$\scriptscriptstyle\circ$}\hss}}#1\else {\accent"17 #1}\fi}

\end{document}